\documentclass[12pt,a4paper]{amsart}
\usepackage{fullpage,setspace}
\usepackage[T1]{fontenc}
\usepackage{lmodern}
\usepackage{amssymb}
\usepackage{mathrsfs}
\usepackage{tikz}
\usepackage{tikz-cd}
\usetikzlibrary{arrows.meta}
\usepackage[numbers,sort&compress]{natbib}
\definecolor{dark-red}{rgb}{0.5,0.15,0.15}
\usepackage[colorlinks=true,linkcolor=black,citecolor=dark-red,urlcolor=dark-red]{hyperref}
\hypersetup{pdftitle={The homotopy types of directed path and trace spaces},pdfauthor={Philippe Gaucher}}
\swapnumbers

\makeatletter
\let\leq\@undefined
\let\geq\@undefined
\let\vec\@undefined
\let\phi\@undefined
\let\epsilon\@undefined
\let\injlim\@undefined
\let\projlim\@undefined
\makeatother
\newcommand{\leq}{\leqslant}
\newcommand{\geq}{\geqslant}
\newcommand{\vec}{\overrightarrow}
\newcommand{\phi}{\varphi}
\newcommand{\epsilon}{\varepsilon}
\newcommand{\injlim}{\varinjlim}
\newcommand{\projlim}{\varprojlim}

\newtheorem{theorem}{Theorem}[section]
\newtheorem{proposition}[theorem]{Proposition}
\newtheorem{lemma}[theorem]{Lemma}
\newtheorem{corollary}[theorem]{Corollary}
\theoremstyle{definition}

\newtheorem{remark}[theorem]{Remark}
\newtheorem{notation}[theorem]{Notation}
\newcommand{\I}{[0,1]}
\newcommand{\kd}{k_{\Delta}}
\newcommand{\dP}{\overrightarrow{\mathcal P}}
\newcommand{\dT}{\overrightarrow{\mathcal T}}
\newcommand{\TopD}{\mathsf{Top}_{\Delta}}
\newcommand{\id}{\operatorname{id}}

\title[Directed path and trace spaces]
{The homotopy types of directed path and trace spaces}
\author[P. Gaucher]{Philippe Gaucher}
\address{Universit\'e Paris Cit\'e, CNRS, IRIF, F-75013, Paris, France}
\urladdr{\url{https://www.irif.fr/~gaucher}}

\begin{document}

\begin{abstract}
We construct a saturated directed space with a Hausdorff $\Delta$-generated underlying space and two distinct points such that the trace space between them is homeomorphic to a square, whereas the directed path space has a nontrivial fundamental group. In particular, the canonical quotient map is not a weak homotopy equivalence. The same conclusion holds for regular directed paths modulo increasing homeomorphisms.
\end{abstract}
\maketitle
\setcounter{tocdepth}{1}
\tableofcontents
\hypersetup{linkcolor = dark-red}

\section{Introduction}

\subsection*{Presentation} Directed spaces are topological spaces equipped with a distinguished collection of continuous paths, called \emph{directed paths}, which contains all constant paths and is closed under continuous nondecreasing reparametrization and concatenation \cite{DAT_book}. They provide geometric models of concurrent processes, notably in cubical and globular settings.

In the cubical setting, geometric realizations of precubical sets carry directed paths obtained by concatenating paths that are nondecreasing in each cubical coordinate \cite{DAT_book}. In the globular setting, the directed paths associated with a cellular multipointed $d$-space are constant paths or nondecreasing reparametrizations of restrictions of execution paths to subintervals \cite[Theorem~4.9]{GlobularNaturalSystem}.

A trace is an equivalence class of directed paths under the relation generated by continuous nondecreasing reparametrizations that preserve both endpoints. For the geometric realization $X=|K|$ of an arbitrary precubical set $K$, the canonical quotient
\[
q_{u,v}: \dP(X)(u,v)\longrightarrow\dT(X)(u,v)
\]
is a homotopy equivalence for every pair $u,v\in X$ \cite[Proposition~2.16(3)]{MR2521708}. For a $q$-cofibrant multipointed $d$-space $Y$, in particular for a cellular one, the analogous quotient of the execution path space is a homotopy equivalence for endpoints $u,v\in Y^0$, where $Y^0$ is the distinguished set of states; see \cite[Theorem~16]{Moore3}.

To the author's knowledge, it was previously unknown whether this quotient must be a weak homotopy equivalence for every directed space. Our main result is the following.

\begin{samepage}
\begin{theorem}\label{thm:main}
	There exist a saturated directed space $X$ with a Hausdorff $\Delta$-generated underlying space and distinct points $u,v\in X$ such that $\dT(X)(u,v)$ is homeomorphic to $[-2,2]^2$ and $\dP(X)(u,v)$ has a nontrivial fundamental group at a suitable basepoint. In particular, the canonical quotient
	\[
	q_{u,v}: \dP(X)(u,v)\longrightarrow\dT(X)(u,v)
	\]
	is not a weak homotopy equivalence.
\end{theorem}
\end{samepage}

Thus this property does not follow from the axioms of directed spaces alone, even when the directed structure is saturated and the underlying space is Hausdorff (and hence $T_1$).

The example shows that directed path spaces and trace spaces can have different weak homotopy types even when they arise from the same directed space. For applications to computer science, this raises the question of which homotopy type captures the intended behavior: that of directed paths or that of traces, in which changes of speed and waiting time are declared unobservable?

\subsection*{Geometric idea} Start with a family $B\times I$ of directed intervals $I=[0,1]$, indexed by the square $B=[-2,2]^2$. Replace the midpoint of each fiber indexed by a point of the unit circle $A\subset B$ by an interval. Figure~\ref{fig:F-three-dimensional} depicts the resulting family $F$. Then collapse all lower endpoints to one point and all upper endpoints to another, so that the resulting intervals share exactly their two endpoints. 

Before these identifications, the family admits continuous fiberwise parametrizations over compact subsets of $B$ meeting $A$ in at most countably many points (Lemma~\ref{lem:countable}). These parametrizations yield lifts of enough curves to identify the trace space with the square (Proposition~\ref{prop:trace-square}). On the other hand, a filling of the boundary loop \eqref{eq:boundary-loop} in the directed path space would force the evaluation map to cover every fiber. Restricting this map over the closed slice \eqref{eq:closed-slice} would give a continuous surjection from a compact metrizable space onto a space with uncountably many pairwise disjoint nonempty open subsets; see \eqref{eq:manyopens}. Their inverse images would contradict second countability.

The space $F$ belongs to the class of split compact spaces described in \cite[Definition~2.1]{Koszmider16}. The underlying construction method is not new: here the points of $A\times\{1/2\}\subset B\times I$ are replaced by intervals.

\subsection*{Regular directed paths} The boundary loop is a family of regular directed paths, and the parametrizations used to establish the quotient topology are regular as well: none has a nontrivial interval of constancy. Consequently, the quotient map from regular directed paths to their classes under increasing homeomorphisms need not be a weak homotopy equivalence (Corollary~\ref{cor:regular}). This also holds for the ordinary compact-open topology.

In \cite[Remark~4.6]{reparam}, the quotient map $Q$, corresponding to our $q_{x,y}$, is only asserted to induce surjections on homotopy groups. Our example does not contradict this assertion, since its trace space is contractible. The separate quotient map $Q_R$ from regular directed paths to regular directed traces is marked $\simeq$ in the same diagram, on p.~114. Corollary~\ref{cor:regular} contradicts this weak homotopy equivalence assertion for $Q_R$. The corresponding assertion for ordinary, not necessarily directed, regular paths is \cite[Corollary~3.5]{reparam}, whose proof infers a fibration from the freeness of the reparametrization action. Freeness alone does not imply that an orbit map is a fibration. The homeomorphism between the two trace spaces in \cite[Corollary~4.5]{reparam} is not contradicted by our example.

\section{Definitions}

Write $I=\I$. We use the French convention: compact means quasi-compact and Hausdorff. Products and subspaces have their ordinary topologies unless $\Delta$-kelleyfication is explicitly indicated.

The category of $\Delta$-generated topological spaces is the final closure of the full subcategory of standard topological simplices in the category of ordinary topological spaces. The right adjoint of the inclusion into ordinary topological spaces is the $\Delta$-kelleyfication $\kd$. Thus $\kd Y$ has the final topology for the continuous maps $I\to Y$, and $Y$ is $\Delta$-generated when $\kd Y=Y$. The category of $\Delta$-generated spaces is denoted by $\TopD$. We use the coreflection property of $\kd$: for a $\Delta$-generated source $K$, a map $K\to Y$ is continuous if and only if it is continuous as a map $K\to\kd Y$ \cite{FR}.

A \emph{directed space} in Grandis's sense is a space equipped with a set of continuous paths containing all constant paths and closed under concatenation and continuous nondecreasing reparametrization \cite[\S1.1, p.~284]{mg}. We use \emph{saturated} in the sense of \cite[Definition~4.3]{reparam}: if $x: I\to X$ is continuous and $\phi: I\to I$ is a continuous nondecreasing surjection, then directedness of $x\phi$ implies directedness of $x$.

A path is called \emph{regular} if it is constant or has no nontrivial stop interval \cite[Definition~1.1]{reparam}.

For $u,v\in X$, the \emph{directed path space} $\dP(X)(u,v)$ is the set of directed paths from $u$ to $v$, equipped with the $\Delta$-kelleyfication of the subspace topology inherited from the compact-open topology on $C(I,X)$. The \emph{trace space} $\dT(X)(u,v)$ is its quotient in $\TopD$ by the equivalence relation generated by $x\sim x\phi$, where $\phi: I\to I$ ranges over the continuous nondecreasing surjections.

We will repeatedly use the compact-open exponential law and continuity of evaluation \cite[Theorems~46.10--46.11]{zbMATH01461253}, together with the coreflection property recalled above. Specifically, if $X=\kd Z$ and $K$ is a finite polyhedron, a set map $f: K\to\dP(X)(u,v)$ is continuous if and only if its adjoint $(k,t)\mapsto f(k)(t)$ is continuous as a map $K\times I\to Z$. Indeed, $K\times I$ is $\Delta$-generated, so a continuous adjoint into $Z$ factors continuously through $X$. The exponential law then gives continuity into the ordinary directed path space, and the coreflection property for $K$ gives continuity into its $\Delta$-kelleyfication. The converse follows from continuity of evaluation and of the identity $X\to Z$.

\section{A compact family of directed intervals}

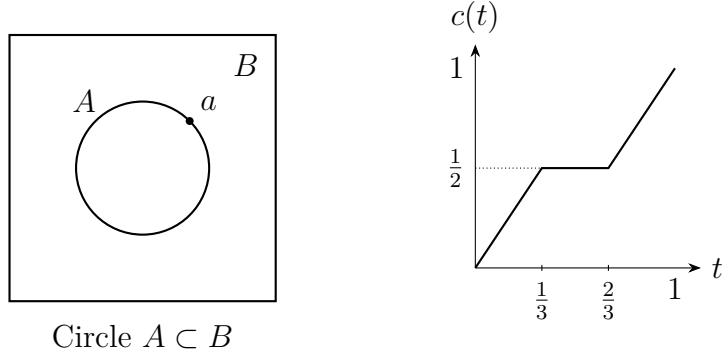
\begin{figure}[htbp]
	\centering
	\begin{tikzpicture}[scale=.88,>=Stealth]
		\draw[thick] (-2,-2) rectangle (2,2);
		\draw[thick] (0,0) circle (1);
		\node at (1.55,1.55) {$B$};
		\node at (-.88,1.) {$A$};
		\fill (.7071,.7071) circle (1.7pt);
		\node[above right] at (.7071,.7071) {$a$};
		\node[below] at (0,-2.2) {Circle $A\subset B$};
		\begin{scope}[shift={(5,-1.5)},x=3cm,y=3cm]
			\draw[->] (0,0) -- (1.13,0) node[right] {$t$};
			\draw[->] (0,0) -- (0,1.12) node[above] {$c(t)$};
			\draw[thick] (0,0)--(1/3,1/2)--(2/3,1/2)--(1,1);
			\draw[densely dotted] (0,1/2)--(1/3,1/2);
			\draw (1/3,.015)--(1/3,-.015) node[below] {$\frac13$};
			\draw (2/3,.015)--(2/3,-.015) node[below] {$\frac23$};
			\node[left] at (0,1/2) {$\frac12$};
			\node[below] at (1,0) {$1$};
			\node[left] at (0,1) {$1$};
		\end{scope}
	\end{tikzpicture}
	\caption{The indexing square $B$, the circle $A$, and the plateau map $c$ used to insert a middle interval in each exceptional fiber.}
	\label{fig:construction}
\end{figure}

Consider the square $B=[-2,2]^2$ and the circle $A=\{a\in\mathbb R^2:\|a\|=1\}$, where $\|\cdot\|$ denotes the Euclidean norm; see Figure~\ref{fig:construction}.

We now construct the compact family of intervals described in the introduction. Consider the projection $B\times I\to B$, with height coordinate $s\in I$. For each $a\in A$, we replace the midpoint of the fiber over $a$ by an interval. The height $s$ is constant on this inserted interval, so we introduce an additional coordinate $t_a$ that records progress along it. Away from the fiber over $a$, this coordinate is uniquely determined by the pair $(b,s)$. In particular, on each fixed fiber over $b\notin A$, the height $s$ determines all the additional coordinates.

Define the continuous nondecreasing surjection $c: I\to I$ by
\begin{equation}\label{eq:plateau}
c(t)=
\begin{cases}
\dfrac32t,&0\leq t\leq\dfrac13,\\[8pt]
\dfrac12,&\dfrac13\leq t\leq\dfrac23,\\[8pt]
\dfrac{3t-1}{2},&\dfrac23\leq t\leq1.
\end{cases}
\end{equation}
For every $a\in A$, let
\begin{equation}\label{eq:ga}
\delta_a(b)=\min\{1,\|b-a\|\},
\qquad
g_a(b,t)=\delta_a(b)t+(1-\delta_a(b))c(t).
\end{equation}
When $b\ne a$, the map $g_a(b,-)$ is a strictly increasing homeomorphism of $I$; when $b=a$, it is $c$. Thus the relation $g_a(b,t_a)=s$ gives a unique $t_a$ when $b\ne a$, whereas $g_a(a,t_a)=1/2$ supplies exactly the desired middle interval.

Define
\begin{equation}\label{eq:F}
F=\left\{(b,s,(t_a)_{a\in A})\in B\times I\times I^A:
g_a(b,t_a)=s\text{ for all }a\in A\right\},
\end{equation}
where $I^A=\prod_{a\in A}I$ has the product topology, and let $p: F\to B$ be the first projection. Since each $g_a$ is continuous, the defining equations make $F$ a closed subspace of $B\times I\times I^A$, which is compact by Tychonoff's theorem. Hence $F$ is compact. Write $F_b=p^{-1}(\{b\})$. Each $F_b$ is compact as well, being a closed subspace of $F$. The fiber $F_b$ is called \emph{exceptional} when $b\in A$.

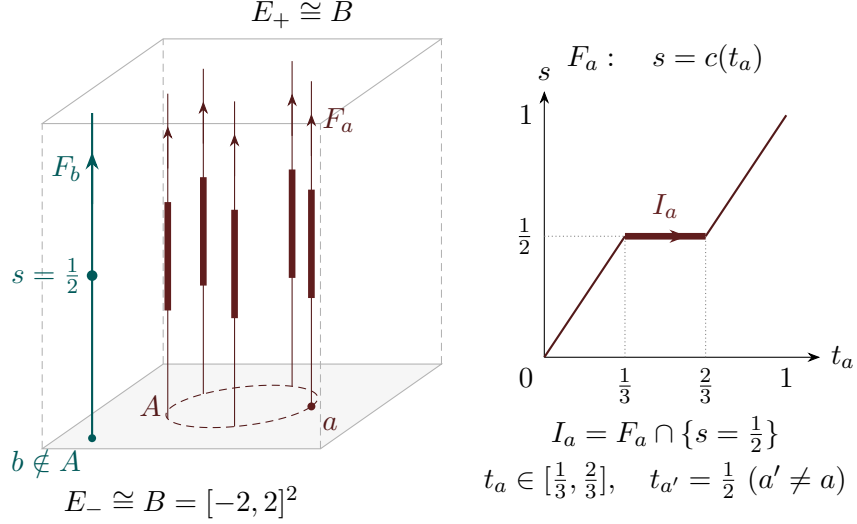
\begin{figure}[htbp]
	\centering
	\begin{tikzpicture}[
		x={(.92cm,0cm)}, y={(.40cm,.28cm)}, z={(0cm,4.3cm)},
		>=Stealth, font=\small,
		edge/.style={gray!60,thin},
		ordinary/.style={teal!70!black,thick},
		exceptional/.style={dark-red!65!black,thin},
		inserted/.style={dark-red!75!black,line width=2.5pt}
		]
		
		\fill[gray!7]
		(-2,-2,0)--(2,-2,0)--(2,2,0)--(-2,2,0)--cycle;
		\draw[edge]
		(-2,-2,0)--(2,-2,0)--(2,2,0)--(-2,2,0)--cycle;
		\foreach \x/\y in {-2/-2,2/-2,2/2,-2/2}
		\draw[edge,densely dashed] (\x,\y,0)--(\x,\y,1);
		\draw[edge]
		(-2,-2,1)--(2,-2,1)--(2,2,1)--(-2,2,1)--cycle;
		
		\draw[dark-red!65!black,densely dashed]
		plot[domain=0:360,samples=81,smooth]
		({cos(\x)},{sin(\x)},0);
		\node[left,dark-red!65!black] at (-1,0,0) {$A$};
		\node[below,yshift=-10pt] at (0,-2,0)
		{$E_-\cong B=[-2,2]^2$};
		\node[above] at (0,2,1) {$E_+\cong B$};
		
		\foreach \angle in {72,144,0,216,288}{
			\draw[exceptional]
			({cos(\angle)},{sin(\angle)},0)--
			({cos(\angle)},{sin(\angle)},1);
			\draw[inserted]
			({cos(\angle)},{sin(\angle)},1/3)--
			({cos(\angle)},{sin(\angle)},2/3);
			\draw[exceptional,->]
			({cos(\angle)},{sin(\angle)},.77)--
			({cos(\angle)},{sin(\angle)},.90);
		}
		\fill[dark-red!75!black] (1,0,0) circle (1.5pt);
		\node[below right,dark-red!65!black] at (1,0,0) {$a$};
		\node[right,dark-red!65!black] at (1,0,.88) {$F_a$};
		
		\draw[ordinary] (-1.5,-1.5,0)--(-1.5,-1.5,1);
		\draw[ordinary,->] (-1.5,-1.5,.72)--(-1.5,-1.5,.88);
		\fill[teal!70!black] (-1.5,-1.5,.5) circle (2pt);
		\fill[teal!70!black] (-1.5,-1.5,0) circle (1.5pt);
		\node[below left,teal!70!black] at (-1.5,-1.5,0)
		{$b\notin A$};
		\node[left,teal!70!black] at (-1.5,-1.5,.83) {$F_b$};
		\node[left,teal!70!black] at (-1.5,-1.5,.5)
		{$s=\tfrac12$};
		
		\begin{scope}[shift={(4cm,.65cm)},x=3.2cm,y=3.2cm]
			\node[above] at (.5,1.13) {$F_a:\quad s=c(t_a)$};
			\draw[->] (0,0)--(1.14,0) node[right] {$t_a$};
			\draw[->] (0,0)--(0,1.10) node[above] {$s$};
			
			\draw[gray,densely dotted] (0,.5)--(1/3,.5);
			\foreach \t in {1/3,2/3}
			\draw[gray,densely dotted] (\t,0)--(\t,.5);
			
			\draw[exceptional,thick]
			(0,0)--(1/3,.5)
			(2/3,.5)--(1,1);
			\draw[inserted] (1/3,.5)--(2/3,.5);
			\draw[dark-red!75!black,->,thick] (.43,.5)--(.59,.5);
			\node[above,dark-red!75!black] at (.5,.53) {$I_a$};
			
			\node[left] at (0,.5) {$\tfrac12$};
			\node[left] at (0,1) {$1$};
			\node[below left] at (0,0) {$0$};
			\node[below] at (1/3,0) {$\tfrac13$};
			\node[below] at (2/3,0) {$\tfrac23$};
			\node[below] at (1,0) {$1$};
			
			\node[below,align=center] at (.5,-.20)
			{$I_a=F_a\cap\{s=\tfrac12\}$\\[3pt]
				$t_a\in[\tfrac13,\tfrac23],\quad
				t_{a'}=\tfrac12\ (a'\ne a)$};
		\end{scope}
	\end{tikzpicture}
	
	\caption{Schematic of $F$, with a few fibers displayed and endpoint slices
		$E_-=\{s=0\}$ and $E_+=\{s=1\}$.
		For $b\notin A$, $s$ parametrizes $F_b$.
		For $a\in A$, the thick segment $I_a$ is inserted at
		$s=\tfrac12$ and parametrized by
		$t_a\in[\tfrac13,\tfrac23]$.
		Vertical position in the left panel indicates fiber order only;
		the drawing is not an embedding of $F$ into $\mathbb R^3$.}
	\label{fig:F-three-dimensional}
\end{figure}

\begin{notation}
	For $h\in I$, we denote the slice $\{(b,s,(t_a)_{a\in A})\in F:s=h\}$ by $\{s=h\}$.
\end{notation}

\begin{lemma}\label{lem:fibers}
Each fiber $F_b$ is homeomorphic to $I$, ordered by $s$ when $b\notin A$ and by $t_b$ when $b\in A$. Its lower endpoint has $s=t_a=0$ for all $a$, and its upper endpoint has $s=t_a=1$ for all $a$. All the coordinates $s,t_a$ are nondecreasing along the fiber.
\end{lemma}
\begin{proof}
If $b\notin A$, every $g_a(b,-)$ is a strictly increasing homeomorphism, so the coordinate $s$ gives a continuous bijection $F_b\to I$. It is a homeomorphism since $F_b$ is compact and $I$ is Hausdorff. Its inverse has coordinates $t_a=g_a(b,-)^{-1}(s)$, all nondecreasing.

If $b=a_0\in A$, the coordinate $t_{a_0}$ instead gives a continuous bijection $F_b\to I$: it determines $s=c(t_{a_0})$, and every remaining coordinate is uniquely determined by the formula $t_a=g_a(b,-)^{-1}(s)$ for $a\ne a_0$. Compactness again makes this bijection a homeomorphism. All these coordinates are nondecreasing as functions of $t_{a_0}$. At $s=1/2$, the coordinate $t_{a_0}$ ranges over $[1/3,2/3]$. The assertions about endpoints follow from $g_a(b,-)^{-1}(\{0\})=\{0\}$ and $g_a(b,-)^{-1}(\{1\})=\{1\}$.
\end{proof}

Consider the endpoint slices
\[
E_-:=\{s=0\},\qquad
E_+:=\{s=1\}.
\]
By Lemma~\ref{lem:fibers}, these slices are, respectively, the sets of lower and upper endpoints of all the fibers, with $t_a=0$ on $E_-$ and $t_a=1$ on $E_+$ for every $a\in A$. The projection $p$ restricts to a homeomorphism from each slice onto $B$. Collapse $E_-$ to a single point $u$ and $E_+$ to a distinct point $v$, leaving every other point unchanged. Denote the quotient map by
\[
\rho: F\longrightarrow Z.
\]
As a subset of $F\times F$, the equivalence relation is
\[
R=\operatorname{diag}_F\cup(E_-\times E_-)\cup(E_+\times E_+),
\]
where $\operatorname{diag}_F$ is the diagonal of $F$. Since $F$ is Hausdorff and $E_-,E_+$ are closed, $R$ is closed in $F\times F$. The quotient by this closed equivalence relation on the compact Hausdorff space $F$ is therefore compact Hausdorff. The continuous coordinates $s$ and $t_a$ descend to continuous coordinates on $Z$, since each is constant on each collapsed set $E_-$ and $E_+$.

\begin{notation}
	We use the same notation $\{s=h\}$ for the corresponding slices of $Z$.
\end{notation}

Set $X=\kd Z$. Since $Z$ is Hausdorff and $\Delta$-kelleyfication refines its topology, $X$ is Hausdorff as well.

\begin{notation}
	Let $J_b=\rho(F_b)$, as a subset of $X$.
\end{notation}

A parametrization $I\to F_b\to Z$ factors continuously through $X$ because $I$ is $\Delta$-generated. Since $X$ is Hausdorff, this makes $J_b$ an embedded interval with its original topology. For $b\ne b'$,
\begin{equation}\label{eq:intersection}
J_b\cap J_{b'}=\{u,v\}.
\end{equation}
The set $F\setminus(E_-\cup E_+)$ is open and saturated for the quotient $\rho$, and every equivalence class in it is a singleton. Thus $\rho$ restricts to a homeomorphism from this set onto $Z\setminus\{u,v\}$, giving a continuous index map
\begin{equation}\label{eq:interior-index}
\beta: Z\setminus\{u,v\}\longrightarrow B.
\end{equation}

\section{The saturated directed structure and its trace topology}

Give each $J_b$ the order transported from $F_b$ along $\rho|_{F_b}$. Declare a continuous path in $X$ to be directed if it is constant or is a nondecreasing path contained in a single $J_b$. Such a path need not traverse all of $J_b$.

\begin{proposition}\label{prop:saturated}
This is a saturated directed structure.
\end{proposition}
\begin{proof}
Constant paths are included, and precomposition by any continuous nondecreasing self-map of $I$ preserves containment and monotonicity. Suppose two nonconstant directed paths are composable. If they lie in distinct intervals, their common endpoint must be $u$ or $v$ by \eqref{eq:intersection}. The first case is impossible because a nonconstant nondecreasing path cannot end at the lower endpoint $u$; the second is impossible because such a path cannot start at the upper endpoint $v$. Thus the two paths lie in the same interval, and their concatenation is nondecreasing. Concatenation with a constant path also preserves the structure.

For saturation, suppose $x: I\to X$ is continuous, $\phi: I\to I$ is a continuous nondecreasing surjection, and $x\phi$ is directed. If $x\phi$ is constant, then so is $x$. Otherwise, surjectivity gives $x(I)=(x\phi)(I)\subseteq J_b$ for some $b$. For $s<t$, choose $r_s\in\phi^{-1}(\{s\})$ and $r_t\in\phi^{-1}(\{t\})$. Monotonicity of $\phi$ implies $r_s<r_t$, whence
\[
x(s)=(x\phi)(r_s)\leq_{J_b}(x\phi)(r_t)=x(t).
\]
Since $J_b$ is an embedded interval, $x$ is a continuous nondecreasing path in it and is therefore directed.
\end{proof}

For the rest of the paper, let $P=\dP(X)(u,v)$. Since $u\ne v$, every element of $P$ is nonconstant and lies in a unique interval $J_b$. It traverses the whole of $J_b$ by the intermediate value theorem. For each $b\in B$, choose a strictly increasing homeomorphism $r_b: I\to J_b$; no continuity in $b$ is assumed. Every path in $P$ with image $J_b$ then has the form $r_b\phi$, where $\phi$ is a continuous nondecreasing surjection. Therefore all such paths have the same trace. Paths belonging to distinct intervals cannot be equivalent, since their images differ and surjective reparametrization preserves the image.

The trace set is thus identified with $B$. Let
\begin{equation}\label{eq:Q}
Q: P\longrightarrow B
\end{equation}
be the set map assigning to a path its interval index.

\begin{lemma}\label{lem:Q-continuous}
The map $Q$ is continuous for the usual topology of $B$.
\end{lemma}
\begin{proof}
For $t\in I$, the set
\[
V_t=\{x\in P:x(t)\notin\{u,v\}\}
\]
is open, by continuity of evaluation. These sets cover $P$. On $V_t$, the map $Q$ is the composite of evaluation at $t$, the identity $X\to Z$, and the map $\beta$ in \eqref{eq:interior-index}. Thus $Q$ is continuous on an open cover of its domain.
\end{proof}

\begin{lemma}\label{lem:countable}
	If $C\subset B$ is compact and $C\cap A$ is countable, put $F_C=p^{-1}(C)$.
	Then there is a homeomorphism over $C$
	\[
	\Theta_C: F_C\xrightarrow{\cong}C\times I
	\]
	that is strictly increasing on every fiber and preserves both endpoints.
\end{lemma}

\begin{proof}
	Since $C$ is closed in $B$, the space $F_C$ is compact. Enumerate the distinct elements of $C\cap A$ as $(a_n)$, with a finite or empty indexing set when appropriate. Choose a summable family of positive weights $(w_n)$ with the same indexing set. Define the function $h_C: F_C\to I$ by
	\begin{equation}\label{eq:weighted-height}
		h_C=\frac{s+\sum_n w_n t_{a_n}}{1+\sum_n w_n}.
	\end{equation}
	Since $0\leq t_{a_n}\leq 1$, the series converges uniformly, so $h_C$ is continuous.

	By Lemma~\ref{lem:fibers}, all the coordinates appearing in \eqref{eq:weighted-height} are nondecreasing on every fiber. Let $x<y$ be two distinct points of a fiber $F_b$, with $b\in C$. If $s(x)<s(y)$, then $h_C(x)<h_C(y)$. Otherwise, $s(x)=s(y)$, and Lemma~\ref{lem:fibers} implies that $b=a_k$ for some $k$, that $x$ and $y$ belong to the inserted middle interval, and that $t_{a_k}(x)<t_{a_k}(y)$. Since $w_k>0$, we again obtain $h_C(x)<h_C(y)$. Thus $h_C$ is strictly increasing on every fiber. Its values at the lower and upper endpoints are respectively $0$ and $1$.

	Since each fiber is an interval, the intermediate value theorem shows that $h_C$ maps it bijectively onto $I$. Thus $\Theta_C=(p,h_C)$ is a continuous bijection. The space $F_C$ is compact, and $C\times I$ is Hausdorff, so $\Theta_C$ is a homeomorphism. It is strictly increasing on every fiber and preserves both endpoints by construction. If $C\cap A=\varnothing$, then $\Theta_C=(p,s)$.
\end{proof}

\begin{lemma}\label{lem:curve-lifting}
	If $\alpha: I\to B$ is continuous and $\alpha(I)\cap A$ is countable, then there exists a continuous map $\widetilde\alpha: I\to P$ with $Q\widetilde\alpha=\alpha$.
\end{lemma}
\begin{proof}
	The set $C=\alpha(I)$ is compact. Using Lemma~\ref{lem:countable}, define the continuous map
	\[
	e: I\times I\longrightarrow Z,
	\qquad e(r,t)=\rho\bigl(\Theta_C^{-1}(\alpha(r),t)\bigr).
	\]
	For each $r$, the path $t\mapsto e(r,t)$ is strictly increasing from $u$ to $v$ in $J_{\alpha(r)}$. The exponential law gives the required continuous map into $P$.
\end{proof}

\begin{proposition}\label{prop:trace-square}
The map $Q$ is an ordinary topological quotient onto the square $B$ with its usual topology. Consequently, the quotient in $\TopD$ is also $B$:
\[
\dT(X)(u,v)\cong B.
\]
\end{proposition}
\begin{proof}
The usual topology of the square can be tested by curves obtained by joining a convergent sequence with straight segments. Such curves meet $A$ in at most countably many points and hence lift by Lemma~\ref{lem:curve-lifting}. Here is the precise argument.

The map $Q$ is surjective because each $J_b$ admits a directed parametrization. By Lemma~\ref{lem:Q-continuous}, it remains to prove that if $Q^{-1}(U)$ is open, then $U$ is open. If $U\subset B$ is not open, there exist $b\in U$ and a sequence $(b_n)_{n\geq1}$ in $B\setminus U$ with $b_n\to b$. Define $\alpha(0)=b$, $\alpha(1/n)=b_n$, and let $\alpha$ be affine on each $[1/(n+1),1/n]$. Convexity of $B$ keeps the curve in $B$, and convergence of $(b_n)$ gives continuity at $0$.

Each nondegenerate straight segment meets the circle $A$ in at most two points; a degenerate segment contributes at most one. The image of $\alpha$ is the union of these countably many segments and $\{b\}$. Thus $\alpha(I)\cap A$ is countable. By Lemma~\ref{lem:curve-lifting}, $\alpha$ has a continuous lift to $P$. If $Q^{-1}(U)$ were open, then
\[
\alpha^{-1}(U)
=\widetilde\alpha^{-1}(Q^{-1}(U))
\]
would be an open subset of $I$ containing $0$ and excluding all $1/n$, which is impossible.

The equivalence classes are exactly the fibers of $Q$, as shown above. The usual square is $\Delta$-generated, so this ordinary quotient already lies in $\TopD$ and has the required categorical universal property.
\end{proof}

\section{A nontrivial loop killed by the trace quotient}

The boundary $\partial B$ is disjoint from $A$, so Lemma~\ref{lem:countable} already provides a continuous family of paths indexed by $\partial B$. Our choice of $\delta_a$ makes this loop particularly simple. For $b\in\partial B$ and $a\in A$, the reverse triangle inequality gives
\[
\|b-a\|\geq\|b\|-1\geq1.
\]
Hence $\delta_a(b)=1$ for every $a$, so $g_a(b,t)=t$. In particular, there is an explicit family
\begin{equation}\label{eq:boundary-loop}
\ell: \partial B\longrightarrow P,
\qquad
\ell(b)(t)=\rho\bigl(b,t,(t)_{a\in A}\bigr).
\end{equation}
The adjoint map $\partial B\times I\to Z$ is continuous, and each of its paths traverses an entire interval $J_b$ in a strictly increasing manner. The exponential law gives continuity of $\ell$ into $P$. By construction, $Q\ell$ is the inclusion $\partial B\subset B$:
\begin{equation}\label{eq:boundary-index}
Q\ell=\operatorname{incl}_{\partial B\subset B}.
\end{equation}
Choose $b_*=(2,0)\in\partial B$ and use $\ell(b_*)$ as the basepoint of $P$.

\begin{proof}[Proof of Theorem~\ref{thm:main}]
The space $X$ is Hausdorff and belongs to $\TopD$ by construction, and its directed structure is saturated by Proposition~\ref{prop:saturated}. Proposition~\ref{prop:trace-square} identifies the trace quotient with $Q: P\to B$. Since $B$ is contractible, $Q_*[\ell]=0$. We prove that $[\ell]\ne0$ in $\pi_1(P,\ell(b_*))$ by showing that $\ell$ cannot be filled by a disk of directed paths.

Suppose such a filling existed. It would be a continuous extension $L: B\to P$ with $L|_{\partial B}=\ell$. Put $f=QL$. This gives the commutative diagram
\[
\begin{tikzcd}[column sep=large,row sep=large]
\partial B \arrow[r,"\ell"] \arrow[d,hook]
  & P \arrow[d,"Q"] \\
B \arrow[r,"f"'] \arrow[ur,"L",dashed]
  & B .
\end{tikzcd}
\]
Equation~\eqref{eq:boundary-index} implies $f|_{\partial B}=\id$. Every such map $f$ is surjective: if it omitted an interior point, radial projection from that point, composed with $f$, would retract $B$ onto $\partial B$, contrary to the no-retraction theorem \cite[Corollary~2.15]{MR1867354}.

Consider the map
\begin{equation}\label{eq:surjective-evaluation}
	e_L: B\times I\longrightarrow Z,
	\qquad e_L(b,t)=L(b)(t).
\end{equation}
It is continuous as the composite of $L\times\id_I$, evaluation into $X$, and the identity $X\to Z$. It is surjective: for every $c\in B$, choose $b$ with $f(b)=c$. The directed path $L(b)$ traverses the whole of $J_c$. Since the intervals $J_c$ cover the underlying set of $Z$, every point lies in the image of $e_L$.

Consider the closed middle slice
\begin{equation}\label{eq:closed-slice}
	M=\{s=1/2\}\subset Z.
\end{equation}
If $b\ne a$, strict monotonicity of $g_a(b,-)$ and the equality $g_a(b,1/2)=1/2$ imply $t_a=1/2$ on $M\cap J_b$. Over $b=a$, the coordinate $t_a$ ranges over $[1/3,2/3]$. Define
\begin{equation}\label{eq:manyopens}
	U_a=M\cap\{t_a>3/5\},\qquad a\in A.
\end{equation}
Each $U_a$ is open in $M$ because $t_a$ is continuous. Since $1/2<3/5<2/3$, it is nonempty and contained in $J_a$. The intervals $J_a$ meet only at $u$ and $v$, neither of which belongs to $M$, so these open sets are pairwise disjoint. Let
\[
D=e_L^{-1}(M)\subset B\times I.
\]
The space $D$ is closed in $B\times I$, hence compact metrizable and second countable. Since $e_L$ is surjective, the sets
\[
(e_L|_D)^{-1}(U_a),\qquad a\in A,
\]
are pairwise disjoint nonempty open subsets of $D$. Since $A$ is uncountable, this contradicts the fact that a second countable space has at most countably many pairwise disjoint nonempty open subsets.

Thus $\ell$ is not nullhomotopic. Its image under $Q$ is the boundary loop in the contractible square, so $Q_*[\ell]=0$. Hence $Q_*$ is not injective on $\pi_1$, and $Q$ is not a weak homotopy equivalence.
\end{proof}

\begin{remark}\label{rem:ordinary-topology}
The conclusion of Theorem~\ref{thm:main} also holds for the directed path space $P_{\mathrm{co}}$ with its ordinary compact-open topology. Indeed, since disks and spheres are $\Delta$-generated, the canonical map $\kd U\to U$ is a weak homotopy equivalence for every space $U$. In particular, $P=\kd P_{\mathrm{co}}\to P_{\mathrm{co}}$ induces isomorphisms on homotopy groups. The index map $Q_{\mathrm{co}}: P_{\mathrm{co}}\to B$ is continuous by the evaluation argument of Lemma~\ref{lem:Q-continuous}, and it is a quotient map because its composite with $P\to P_{\mathrm{co}}$ is the quotient map $Q$.
\end{remark}

\begin{corollary}\label{cor:regular}
Let $P_{\mathrm{reg}}$ be the space of regular directed paths from $u$ to $v$, equipped with the $\Delta$-kelleyfication of the compact-open topology. The quotient by the group $\operatorname{Homeo}_+(I)$ of increasing homeomorphisms of $I$ is homeomorphic to $B$, and the quotient map
\[
Q_{\mathrm{reg}}: P_{\mathrm{reg}}\longrightarrow
P_{\mathrm{reg}}/\operatorname{Homeo}_+(I)\cong B
\]
is not a weak homotopy equivalence. These conclusions also hold with the ordinary compact-open topology and the ordinary topological quotient.
\end{corollary}
\begin{proof}
Since $u\ne v$, a directed path from $u$ to $v$ is regular precisely when it is strictly increasing on its interval $J_b$. Such a path is a homeomorphism from $I$ onto $J_b$. Any two such paths in $J_b$ therefore differ by an increasing homeomorphism of $I$, so the orbits are exactly the fibers of $Q_{\mathrm{reg}}=Q|_{P_{\mathrm{reg}}}$.

The lifts in Lemma~\ref{lem:curve-lifting} take values in $P_{\mathrm{reg}}$ and are continuous for its topology by the same exponential-law argument. The proof of Proposition~\ref{prop:trace-square} therefore applies to $Q_{\mathrm{reg}}$, giving the quotient topology of $B$. Similarly, the loop $\ell$ of \eqref{eq:boundary-loop} is continuous into $P_{\mathrm{reg}}$. A filling in $P_{\mathrm{reg}}$ would give a filling in $P$ by the continuous inclusion, contradicting the proof of Theorem~\ref{thm:main}. Thus $Q_{\mathrm{reg}}$ kills a nonzero fundamental-group element.

The ordinary compact-open case follows from the same quotient-map argument and the invariance of homotopy groups under $\Delta$-kelleyfication, as in Remark~\ref{rem:ordinary-topology}.
\end{proof}


\begin{thebibliography}{10}

	\bibitem{reparam}
	U.~Fahrenberg and M.~Raussen.
	\newblock Reparametrizations of continuous paths.
	\newblock {\em J. Homotopy Relat. Struct.}, 2(2):93--117, 2007.
	\newblock \url{https://eudml.org/doc/233713}.

	\bibitem{DAT_book}
	L.~Fajstrup, E.~Goubault, E.~Haucourt, S.~Mimram, and M.~Raussen.
	\newblock {\em Directed algebraic topology and concurrency}.
	\newblock SpringerBriefs Appl. Sci. Technol. Springer, 2016.
	\newblock \url{https://doi.org/10.1007/978-3-319-15398-8}.

	\bibitem{FR}
	L.~Fajstrup and J.~Rosick{\'y}.
	\newblock A convenient category for directed homotopy.
	\newblock {\em Theory Appl. Categ.}, 21:7--20, 2008.

	\bibitem{Moore3}
	P.~Gaucher.
	\newblock Homotopy theory of {M}oore flows ({III}).
	\newblock {\em North-West. Eur. J. Math.}, 10:55--113, 2024.

	\bibitem{GlobularNaturalSystem}
	P.~Gaucher.
	\newblock Natural homotopy of multipointed $d$-spaces.
	\newblock {\em Math. J. Okayama Univ.}, 68:13--62, 2026.

	\bibitem{mg}
	M.~Grandis.
	\newblock Directed homotopy theory. {I}.
	\newblock {\em Cah. Topol. G\'eom. Diff\'er. Cat\'eg.}, 44(4):281--316, 2003.

	\bibitem{MR1867354}
	A.~Hatcher.
	\newblock {\em Algebraic topology}.
	\newblock Cambridge University Press, Cambridge, 2002.
	\newblock \url{https://pi.math.cornell.edu/~hatcher/AT/ATpage.html}.

	\bibitem{Koszmider16}
	P.~Koszmider.
	\newblock On the problem of compact totally disconnected reflection of nonmetrizability.
	\newblock {\em Topology Appl.}, 213:154--166, 2016.
	\newblock \url{https://doi.org/10.1016/j.topol.2016.08.017}.

	\bibitem{zbMATH01461253}
	J.~R. Munkres.
	\newblock {\em Topology}.
	\newblock Prentice Hall, Upper Saddle River, NJ, 2nd edition, 2000.

	\bibitem{MR2521708}
	M.~Raussen.
	\newblock Trace spaces in a pre-cubical complex.
	\newblock {\em Topology Appl.}, 156(9):1718--1728, 2009.
	\newblock \url{https://doi.org/10.1016/j.topol.2009.02.003}.

\end{thebibliography}

\end{document}